\documentclass[11 pt]{amsart}
\usepackage[top=3cm, bottom=3cm,left=3cm, right=3cm]{geometry}

\usepackage{amsmath}
\usepackage{amsfonts}
\usepackage{amssymb}

\date{}

\newcommand{\beqa}{\begin{eqnarray*}}
\newcommand{\eeqa}{\end{eqnarray*}}
\newcommand{\beqn}{\begin{eqnarray}}
\newcommand{\eeqn}{\end{eqnarray}}

\newcommand{\R}{\mathbb R}

\newcommand{\N}{\mathbb N}

\newcommand{\al}{\alpha}
\newcommand{\be}{\beta}

\newcommand{\de}{\delta}

\newcounter{cnt1}
\newcounter{cnt2}
\newcounter{cnt3}
\newcommand{\blr}{\begin{list}{$($\roman{cnt1}$)$}
 {\usecounter{cnt1} \setlength{\topsep}{0pt}
 \setlength{\itemsep}{0pt}}}
\newcommand{\bla}{\begin{list}{$($\alph{cnt2}$)$}
 {\usecounter{cnt2} \setlength{\topsep}{0pt}
 \setlength{\itemsep}{0pt}}}
\newcommand{\bln}{\begin{list}{$($\arabic{cnt3}$)$}
 {\usecounter{cnt3} \setlength{\topsep}{0pt}
 \setlength{\itemsep}{0pt}}}
\newcommand{\el}{\end{list}}

\newtheorem{thm}{Theorem}[section]
\newtheorem{lem}[thm]{Lemma}
\newtheorem{cor}[thm]{Corollary}
\newtheorem{ex}[thm]{Example}

\newtheorem{Def}[thm]{Definition}
\newtheorem{Prop}[thm]{Proposition}
\newtheorem{rem}[thm]{Remark}
\newcommand{\Rem}{\begin{rem} \rm}
\newcommand{\bdfn}{\begin{Def} \rm}
\newcommand{\edfn}{\end{Def}}
\newcommand{\TFAE}{the following assertions are equivalent: }

\newcommand{\ba}{\begin{array}}
\newcommand{\ea}{\end{array}}

\usepackage{tikz}

\tikzstyle{vertex}=[scale=0.9,auto=left,circle,fill=black!10,inner
sep=0.9pt]
\usetikzlibrary {positioning}
\usetikzlibrary{arrows}
\usepackage{graphics}
\usepackage{graphicx}
\usepackage{hyperref}
\usepackage{graphicx,epsfig}
\usepackage{tikz}
\usepackage{caption}
\usepackage{subcaption}
\usepackage{float}
\usetikzlibrary{shapes,arrows}

\begin{document}

\title[Weight-partially greedy bases and Weight-Property $(A)$]
{Weight-partially greedy bases and Weight-Property $(A)$}

\author[D. Khurana]{Divya Khurana}
\address[Divya Khurana]{Department of Mathematics\\ The Weizmann Institute of Science,
Rehovot\\ Israel, \textit{E-mail~:}
\textit{divya.khurana@weizmann.ac.il, divyakhurana11@gmail.com}}



\begin{abstract}
In this paper, motivated by the notion of $w$-Property $(A)$ defined
in \cite{BDKOW}, we introduce the notions of  $w$-left Property
$(A)$ and  $w$-right Property $(A)$. We also introduce the notions
of $w$-partially greedy basis (using a characterization of partially
greedy basis from \cite{DK}) and $w$-reverse partially greedy basis.
The main aim of this paper is to study $(i)$ some characterizations
of $w$-partially greedy and $w$-reverse partially greedy basis
$(ii)$ conditions on the weight sequences when $w$-left Property
$(A)$ and (or) $w$-right Property $(A)$ implies $w$-Property $(A)$.
\end{abstract}

\thanks {This work was supported by the Israel
Science Foundation (Grant Number--137/16).}

\maketitle

\section{Introduction}

Let $X$ be a real Banach space with a normalized (Schauder) basis
$(e_n)$ and biorthogonal functionals $(e_n^*)$. If $A\subset\N$ then
$|A|$ denotes the cardinality of $A$ and $1_A=\sum_{i\in A} e_i$.
For $A,B\subset \N$ we write $A<B$ if $\max A<\min B$, if $x\in X$
then we write $supp(x)=\{n:e_n^*(x)\not=0\}$, $P_A(x)=\sum_{i\in A}
e_i^*(x)e_i$ and $P_{A^c}(x)=I-P_A(x)$. Set $\N^m=\{A \subset
\N:~|A|=m\}$, $\N^{< \infty}=\cup_{m=0}^{\infty}\N^m$,
$\mathcal{A}=\{(\varepsilon_i): \varepsilon_i=1~ \mbox{or}~
\varepsilon_i=-1 \mbox{~for ~all}~i\}$ and $1_{\varepsilon
A}=\sum_{i\in A}\varepsilon_i e_i$ for $A\subset \N$,
$\varepsilon\in \mathcal{A}$.

The \textbf{Thresholding Greedy Algorithm (TGA)} was introduced by Konyagin and Temlyakov (\cite{KT}).
This algorithm is defined as follow. For $x\in X$ and $m\in\N$, let $\Lambda_m(x)$ be the set of any $m$-indices such that
\begin{align*}
\underset{n\in \Lambda_m(x)}{\min}|e_n^*(x)|\geq \underset{n\not\in
\Lambda_m(x)}{\max}|e_n^*(x)|.
\end{align*}

Then $G_m(x)=\sum_{n\in \Lambda_m(x)}e_n^*(x)e_{n}$ is called an $m$-th greedy approximant to $x$.

Konyagin and Temlyakov in \cite{KT} defined a basis $(e_n)$ to be
$C$-\textbf{greedy} if for all $x\in X$, $m\in \N$, we have
\begin{align*}
\|x-G_m(x)\| \leq C\inf~\left\{ \left
\|x-\underset{i\in A}{\sum}a_ie_i\right\|:~|A|=m,~a_i\in \R,~i\in
A\right\}.
\end{align*}

They showed that greedy bases are characterized by unconditionality and democratic property. Recall that a basis $(e_n)$ of
a Banach space $X$ is unconditional if any rearrangement of the series
$x=\underset{n\geq1}{\sum}e_n^*(x)e_n$ converges in norm to $x$ for all $x\in X$.
A basis $(e_n)$  is said to be democratic if there exists a constant $C \geq 1$ such that $\|1_A\|\leq C \|1_B\|$ where $A,B\in \N^{<\infty}$
and $|A|\leq |B|$.

Konyagin and Temlyakov (\cite{KT}) also introduced the notion of
\textbf{quasi-greedy} basis. A basis $(e_n)$ of a Banach space $X$
is said to be $C$-quasi-greedy if $\|G_m(x)\|\leq C\|x\|$ for all
$x\in X$ and $m\in \N$. Let $C_q$ be the least constant among all
such $C$. Later, Wojtaszczyk \cite{W} proved that a basis is
quasi-greedy if $G_m(x)\longrightarrow x$ as $m\longrightarrow
\infty$ for  $x\in X$.

Dilworth et al. \cite{S} introduced  the notions of  \textbf{almost greedy} basis and \textbf{partially
greedy} basis. A basis $(e_n)$ of a Banach space $X$ is said to
be \textbf{$C$-almost greedy} if
\begin{align*}
\|x-G_m(x)\| \leq C \inf \{\|x-P_A(x)\|:~|A|\leq m\}
\end{align*}
for all $x\in X$, $m\in \N$.

A basis $(e_n)$ is said to be \textbf{$C$-partially greedy}  if for all $x\in X $ and $m\in \N$, we have
\begin{align}\label{defpartial1}
\|x-G_m(x)\| \leq C
\left\|\sum_{n=m+1}^{\infty}e_i^*(x)e_i\right\|.
\end{align}
In \cite{S} the authors proved that  almost greedy
bases  are characterized by quasi-greediness and democracy and partially greedy bases are characterized by
quasi-greediness and the \textbf{conservative} property.  A basis $(e_n)$  of a
Banach space is said to be conservative if there exists a constant
$C\geq 1$ such that $\|1_A\|\leq C\|1_B\|$ whenever $A<B$
and $|A|\leq |B|$.

If $x\in X$ then a \textbf{greedy ordering} for $x$  is a $1-1$ map $\rho:\N\longrightarrow \N$ such
that $supp(x) \subset \rho(\N)$ and if $j<k$, then
$|e_{\rho(j)}^*(x)|\geq |e_{\rho(k)}^*(x)|$. Note that if $x$ has infinite support and the nonzero basis coefficients of $x$  are distinct then $x$ has a unique greedy ordering.

If $\rho$ is a greedy ordering for $x\in X$, then
$\Lambda_m(x)=\{\rho(1),\cdots, \rho(m)\}$ and we denote $\al_m(x)=
\min \Lambda_m(x)$, $\be_m(x)= \max \Lambda_m(x)$.

Recently, in \cite{DK}, the authors gave another characterization of
partially greedy basis and introduced the notion of reverse
partially greedy basis. They proved that a basis $(e_n)$ is
\textbf{partially greedy} if and only if there exists a constant $C$
such that
\begin{align}\label{defpartial2}
\| x-G_m(x)\|\leq C \inf\{ \| x-P_A (x) \|:~|A|\leq
m,~A<\al_m(x)\}
\end{align}
for all $x\in X $ and $m\in \N$.

Throughout this paper we will use this characterization of the partially
greedy bases from \cite{DK}.

A basis $(e_n)$ is said to be \textbf{$C$-reverse partially greedy}  if  for all $x\in X $ and $m\in \N$, we have
\begin{align}\label{defreversepartial}
\| x-G_m(x)\|\leq C \inf\{ \| x-P_A (x) \|:~|A|\leq
m,~A>\be_m(x)\}.
\end{align}

A weight sequence $w=(w_n)$ is a sequence of positive real numbers.
For any $A\subset \N$ we write $w(A)=\sum_{i \in A} w_i$.

In \cite{KPT} the authors studied weight-greedy bases. Recently, in \cite{DKTW} and \cite{BDKOW}, the authors studied  weight-almost greedy bases and weight-partially greedy bases (using the definition from \cite{S}). In \cite{KPT},  \cite{DKTW} and \cite{BDKOW} the authors proved a criterion for weight-greedy, weight-almost greedy and weight-partially greedy bases similar to the one for greedy, almost greedy and partially greedy bases respectively. In \cite{BDKOW} the notion of $w$-Property $(A)$ was introduced and a characterization of $(i)$ $w$-greedy bases in terms of unconditionality and $w$-Property $(A)$ $(ii)$ $w$-almost greedy bases in terms of quasi-greediness and $w$-Property $(A)$ was proved.

In this paper we will introduce the notion of $w$-partially greedy
bases (using the characterization~(\ref{defpartial2})) and
$w$-reverse partially greedy bases. In section 2 we will prove that
a basis is $w$-partially greedy ($w$-reverse partially greedy) if
and only if it is both quasi-greedy and $w$-conservative
($w$-reverse conservative). In section 3 we will  introduce the
notion of $w$-left Property $(A)$ and $w$-right Property $(A)$. We
will prove another characterization of  $w$-partially greedy
($w$-reverse partially greedy) bases, precisely, we will prove that
a basis is $w$-partially greedy ($w$-reverse partially greedy) if
and only if it is quasi-greedy and satisfies $w$-left Property $(A)$
($w$-right Property $(A)$). In this section, we will also consider
some conditions on the weight sequences such that any basis
satisfying both $w$-left Property $(A)$ and $w$-right Property $(A)$
satisfies $w$-Property $(A)$. In the last section we will study
$w$-conservative and $w$-reverse conservative basis.
$w$-conservative and $w$-reverse conservative properties are weaker
conditions than $w$-left Property $(A)$ and $w$-right Property $(A)$
respectively.

Let
\begin{align}\label{partially greedy}
{\stackrel{\sim}\sigma}_{w(\Lambda_m(x))}^L(x)=\inf\{ \| x-P_A (x)
\|:A\in\N^{< \infty}, ~w(A)\leq w(\Lambda_m(x)),~A<\al_m(x)\}
\end{align}

and

\begin{align}\label{reverse partially greedy}
{\stackrel{\sim}\sigma}_{w(\Lambda_m(x))}^R(x)=\inf\{ \| x-P_A (x)
\|:A\in\N^{< \infty}, ~w(A)\leq w(\Lambda_m(x)),~\be_m(x)<A\}
\end{align}

for any $x\in X,~m\in \N$.

\bdfn We say that a basis $(e_n)$ is \textbf{$C$-$w$-partially greedy} if
\begin{align}\label{wpg}
\|x-G_m(x)\|\leq C {\stackrel{\sim}\sigma}_{w(\Lambda_m(x))}^L(x)
\end{align}
for all $x\in X,~m\in \N$. \edfn

\bdfn We say that a basis $(e_n)$ is \textbf{$C$-$w$-reverse partially greedy}  if
\begin{align}\label{wrpg}
\|x-G_m(x)\|\leq C {\stackrel{\sim}\sigma}_{w(\Lambda_m(x))}^R(x)
\end{align}
for all $x\in X,~m\in \N$. \edfn

Let $C_p$, $C_{rp}$ be the least constants satisfying (\ref{wpg}) and (\ref{wrpg}) respectively.

\bdfn A basis $(e_n)$  of a Banach space $X$ is said to be
\textbf{$w$-democratic} if there exists a constant $C\geq 1$ such that
$\|1_A\|\leq C \|1_B\|$ whenever $A,B\in\N^{< \infty}$ and $w(A)\leq w(B)$. \edfn

\bdfn A basis $(e_n)$  of a Banach space $X$ is said to be
\textbf{$w$-conservative} (\textbf{$w$-reverse
conservative}) if there exists a constant $C\geq 1$ such that
$\|1_A\|\leq C \|1_B\|$ whenever $A,B\in\N^{< \infty}$ , $A<B$ ($B<A$) and $w(A)\leq w(B)$. \edfn

If $w=(1,1,1,\cdots)$ then $w(A)=|A|$ for any $A\in \N^{<\infty}$
and in this case $w$-partially greedy ($w$-reverse partially greedy,
$w$-democratic, $w$-conservative, $w$-reverse conservative) basis is
partially greedy (reverse partially greedy, democratic,
conservative, reverse conservative) basis.

We now recall a few results
related to quasi-greedy bases from \cite{S}.
\begin{lem}
Suppose that $(e_n)$ is a $K$-quasi-greedy basis and
$A\in \N^{< \infty}$. Then, for every choice of scalars $(a_i)_{i\in A}$  we have

\begin{align}\label{cuc}
\| \sum_{j\in A} a_je_j\|\leq 2K~ \underset{j\in A}\max
|a_j|\|\sum_{j\in A} e_j\ \|.
\end{align}
\end{lem}

\begin{lem}
Suppose that $(e_n)$ is a $K$-quasi-greedy basis and
$x\in X$ has greedy ordering $\rho$. Then
\begin{align}\label{min}
|e_{\rho(m)}^*(x)| \left\|\sum_1^m e_{\rho(i)}\right\|\leq 4K^2 \| x
\|.
\end{align}
\end{lem}
\section{weight-partially and weight-reverse partially greedy bases}
\begin{thm}\label{pg}
Suppose $(e_n)$ is a basis of a Banach space $X$. Then \TFAE \bla
\item  $(e_n)$ is $w$-partially greedy.
\item  $(e_n)$ is quasi-greedy and $w$-conservative.
\el
\end{thm}
\begin{proof}
We first prove that $(a)$ implies
$(e_n)$ is $w$-conservative. Let $(e_n)$ be $C_p$-$w$-partially greedy, $A,B\in \N^{<\infty}$ with $A<B$ and $w(A)\leq w(B)$.
Consider $x=1_A+(1+\varepsilon)1_B$ for any $\varepsilon>0$. Then
$G_{|B|}(x)=(1+\varepsilon)1_B$. Now from
(\ref{wpg}), we have $\| 1_A \|\leq C_p
{\stackrel{\sim}\sigma}_{w(\Lambda_{|B|}(x))}^L(x)\leq
C_p(1+\varepsilon)\| 1_B \|$. Thus, by letting
$\varepsilon\rightarrow 0$, we get $(e_n)$ is $w$-conservative.
Clearly $(a)$ implies that $\|x-G_m(x)\|\leq C_p \|x\|$ and hence
the basis is quasi-greedy.

Now we will prove that $(b)$ implies $(a)$. Let the  basis $(e_n)$
be $C_q$-quasi-greedy and $w$-conservative with
constant $C$. Let $x\in X$, $m\in \N$ and $G_m(x)=\underset{i\in
\Lambda_m(x)}{\sum}e_i^*(x)e_i$.

If $A\in \N^{<\infty}$ with $A<\al_m(x)$ and $w(A)\leq w(\Lambda_m(x))$, then
\begin{align*}
x-G_m(x)=x-\underset{i\in A}{\sum}e_i^*(x)e_i-\underset{i\in
\Lambda_m(x)}\sum e_i^*(x)e_i+\underset{i\in A}{\sum}e_i^*(x)e_i.
\end{align*}
We can write
\begin{align*}
\underset{i\in
\Lambda_m(x)}{\sum}e_i^*(x)e_i=G_m\left(x-\underset{i\in
A}{\sum}e_i^*(x)e_i\right),
\end{align*} and hence
$$\|
\underset{i\in \Lambda_m(x)}{\sum}e_i^*(x)e_i\|\le
C_q\|x-\underset{i\in A}{\sum}e_i^*(x)e_i\|.$$

From $(\ref{cuc})$, $(\ref{min})$ and the $w$-conservative property we have
\begin{align*}
\|\underset{i\in A} \sum e_i^*(x)e_i  \| &\leq 2C_q~\underset{i\in
A}\max |e_i^*(x)|\| 1_A\|\\ &\leq 2C_q C\underset{i\in
\Lambda_m(x)}{\min}|e_i^*(x)|\|1_{\Lambda_m(x)}\|\\&\leq
8C_q^3C\|\underset{i\in \Lambda_m(x)}{\sum}e_i^*(x)e_i\|\\
& \le 8C_q^4C \| x-\underset{i\in A}{\sum}e_i^*(x)e_i \|
\end{align*}
Thus $\parallel x-G_m(x)\parallel\leq (1+C_q+8C_q^4C)\|
x-\underset{i\in A}{\sum}e_i^*(x)e_i \|$.
\end{proof}

\Rem
In \cite{BDKOW} a basis $(e_n)$ of a Banach space $X$ was defined to be $w$-partially greedy if for all $m,~r$ such that
$w(\{1,\cdots,m\})\leq w(\Lambda_r(x))$, there exists a constant $C$ such that
\begin{align*}
\|x-G_r(x)\|\leq C\|\sum_{m+1}^\infty e_i^*(x)e_i\|
\end{align*}
It was proved in \cite{BDKOW} that a basis is $w$-partially greedy
if and only if it is quasi-greedy and $w$-conservative. Thus from
Theorem~\ref{pg} it follows that the $w$-partially greedy basis
considered in this paper is equivalent to the one considered in
\cite{BDKOW}.
\end{rem}

Similar arguments as in Theorem~\ref{pg} yields the following
result.

\begin{thm}\label{rpg}
Suppose $(e_n)$ is a basis of a Banach space $X$. Then \TFAE \bla
\item  $(e_n)$ is $w$-reverse partially greedy.
\item  $(e_n)$ is quasi-greedy and $w$-reverse conservative.
\el
\end{thm}

\section{weight-left property (A) and weight-right property (A)}

Following notion of $w$-Property $(A)$ was introduced in \cite{BDKOW}.

\bdfn A basis $(e_n)$ of a Banach space $X$ is said to have
$w$-Property $(A)$ if there exists a constant $C\geq 1$ such that
\begin{align}\label{a}
\|x+t1_{\varepsilon A}\|\leq C \|x+t1_{\eta B}\|
\end{align}

for any $x\in X$, $t\geq \sup_j |e_j^*(x)|$,  $A,B\in \N^{<\infty}$,
$A\cap B=\emptyset$, $w(A)\leq w(B)$, $supp(x)\cap (A\cup
B)=\emptyset$ and $\varepsilon,\eta\in \mathcal{A}$. \edfn

Motivated by this notion of $w$-Property $(A)$, we now give the
definitions of $w$-left Property $(A)$ and $w$-right Property $(A)$.

\bdfn We say a basis $(e_n)$ of a Banach space $X$  satisfies
\textbf{$w$-left Property $(A)$} if there exists a constant $C\geq 1$ such that
\begin{align}\label{wlpa}
\|x+t1_{\varepsilon A}\|\leq C \|x+t1_{\eta B}\|
\end{align}

for any $x\in X$, $t\geq \sup_j |e_j^*(x)|$,  $A,B\in \N^{<\infty}$
with $A<B$, $w(A)\leq w(B)$, $supp(x)\cap (A\cup B)=\emptyset$ and
$\varepsilon,\eta\in \mathcal{A}$. \edfn

\bdfn A basis  $(e_n)$ of a Banach space $X$ is said to have \textbf{$w$-right
Property $(A)$} if there exists a constant $C\geq 1$ such that
\begin{align}\label{wrpa}
\|x+t1_{\varepsilon A}\|\leq C \|x+t1_{\eta B}\|
\end{align}

for any $x\in X$, $t\geq \sup_j |e_j^*(x)|$,  $A,B\in \N^{<\infty}$
with $B<A$, $w(A)\leq w(B)$, $supp(x)\cap (A\cup B)=\emptyset$ and
$\varepsilon,\eta\in \mathcal{A}$. \edfn

Let $C_a$, $C_{la}$, $C_{ra}$ be the least constants satisfying (\ref{a}), (\ref{wlpa})
and (\ref{wrpa}) respectively. We will write $(e_n)$ satisfies $C_{a}$-$w$-Property $(A)$, $C_{la}$-$w$-left Property $(A)$
and $C_{ra}$-$w$-right Property $(A)$.

\Rem \label{Aimpliesconservative}
Observe that any basis satisfying
$w$-Property $(A)$, $w$-left Property $(A)$, $w$-right Property $(A)$ is respectively
$w$-democratic, $w$-conservative and $w$-reverse conservative.
\end{rem}

For $w=(1,1,1,\cdots)$ we will write $w$-Property $(A)$ ($w$-left
Property $(A)$, $w$-right Property $(A)$) as Property $(A)$ (left Property
$(A)$, right Property $(A)$). While considering $w$-right Property $(A)$ and $w$-reverse conservative bases
for any $B\in \N^{<\infty}$, we will always include the empty set in the collection $\{A\in \N^{<\infty}: B<A,~w(A)\leq w(B)\}$.

\Rem
Albiac and Wojtaszczyk introduced the following notion of Property $(A)$ (we will refer to
it as classical Property $(A)$) in \cite{AW} to study a characterization of $1$-greedy basis. A basis $(e_n)$ has
classical Property $(A)$ if there exists a constant $C\geq 1$ such that
\begin{align*}
\|x+t1_{\varepsilon A}\|\leq C \|x+t1_{\eta B}\|
\end{align*}
for any $x\in X$, $t\geq \sup_j |e_j^*(x)|$,  $A,B\in \N^{<\infty}$,
$A\cap B=\emptyset$, $|A|=|B|$, $supp(x)\cap (A\cup B)=\emptyset$
and $\varepsilon,\eta\in \mathcal{A}$. In this paper we are dealing
only with the Schauder basis so the Property $(A)$ is equivalent to
the classical Property $(A)$.
\end{rem}

\begin{ex}
We present examples of basis which satisfy left Property $(A)$ or
right Property $(A)$ but does not satisfy Property $(A)$. \bla
\item Let $p=(p_n)$ be a strictly increasing sequence of natural numbers and $M_n(t)=t^{p_n}$ for $t\geq 0$.
Let $X_p$ be the corresponding modular sequence space,
that is, $X_p$ is the Banach space consisting of all sequences
$x=(x_n)$ with $\sum {(\frac {|x_n|}{\lambda}})^{{p_n}}<
\infty$  for some $\lambda >0$ and the norm on $X_p$ is defined as
\begin{align*}
    \|x\|_p=\inf\left\{\lambda >0: \sum {\Big(\frac {|x_n|}{\lambda}}\Big)^{p_n}\leq 1 \right\}.
\end{align*}

Let $x\in X$, $t\geq \sup_j |e_j^*(x)|$,  $A,B\in \N^{<\infty}$ with
$B<A$, $|A|\leq |B|$, $supp(x)\cap (A\cup B)=\emptyset$ and
$\varepsilon,\eta\in \mathcal{A}$.
Then clearly $\|x+t1_{\varepsilon A}\|_p\leq
\|x+t1_{\eta B}\|_p$, thus the canonical basis $(e_n)$ of $X_p$
satisfies right Property $(A)$ and its easy to observe that $(e_n)$
does not satisfy left Property $(A)$.

\item For each $n\in \mathbb{N}$, we define $F_n:=\{A\subset \N:|A|\leq n!,
~n!\leq A\}$ and $F:=\bigcup_{n\geq 1}F_n$.

Observe that the set $F$ is closed under spreading to the right: in
fact, if
 $A,B\in\mathbb{ N}^m$, $A\in F$ and $\min A\leq \min B$, then $B\in F$.

We define a norm on $c_{00}$ (the space of all sequences of real numbers with finitely many non-zero terms) as follows
\begin{align*}
\|x\|=\max \{(|x|,1_A):A\in F\}
\end{align*}
for $x\in c_{00}$. Let $X$ be the completion of $c_{00}$ in this
norm. The canonical basis $(e_n)$ of $X$  is  normalized and
$1$-unconditional.

From the right spreading property of $F$, it follows that $\|x\|\leq
\|y\|$, where $y=\sum a_i e_{n_i}$ is a spread of $x=\sum a_ie_i$
with $n_1<n_2<\ldots$. In particular, if $A<B$ and $|A|\leq |B|$
then for any $x\in X$, $t\geq \sup_j |e_j^*(x)|$, $supp(x)\cap
(A\cup B)=\emptyset$ and $\varepsilon,\eta\in \mathcal{A}$ we get
$\|x+t1_{\varepsilon A}\|\leq \|x+t1_{\eta B}\|$. Thus $(e_n)$
satisfies left Property $(A)$.

Observe that, $\|1_{[1,n!]}\|=(n-1)!$, while
$\|1_{[n!+1,2n!]}\|=n!$. So $(e_n)$ is not a democratic basis and hence
$(e_n)$ cannot satisfy Property $(A)$.

Now consider the dual norm $\|\cdot\|^*$. Since $\|x\|\leq \|y\|$ it
follows easily that $\|x\|^*\geq\|y\|^*$. In particular,
$\|\cdot\|^*$ satisfies right Property $(A)$. But in this case
$\|1_{[1,n!]}\|=(n-1)!$ implies $\|1_{[1,n!]}\|^*\geq n$ while
$\|1_{[n!+1,2n!]}\|^*= 1$
since $[n!+1,2n!]\in F_n$. So the dual norm is not democratic and hence it cannot satisfy Property $(A)$.\\
\el
\end{ex}
\begin{ex}
Now we shall present an example of a basis which is conservative but does not satisfies left Property $(A)$.
Let $(a_i)$ be a sequence of natural numbers such that
\begin{align*}
a_{2i}=i~\mbox{and}~a_{2i-1}=i~ \mbox{for ~i}\in\N.
\end{align*}
We define a norm on $c_{00}$ as follows
\begin{align*}
\|x\|=\sup_N \left|\sum_1^N a_ix_i\right|
\end{align*}
for $x\in c_{00}$. Let $X$ be the completion of $c_{00}$ in this
norm. The canonical basis $(e_n)$ of $X$  is normalized basis.
Clearly $(e_n)$ is conservative but is not a democratic basis.

Let $n\in \N$ be any even number, $A=\{n+1,\cdots,2n\}$ and $B=\{2n+2,2n+4,\cdots,4n-2,4n\}$.
Choose $x\in X$ with $supp(x)=\{2n+1,2n+3,\cdots,4n-3,4n-1\}$ and $x_i=-1$ for all $i \in supp(x)$.
Now for any $\varepsilon,\eta\in \mathcal{A}$ with $\varepsilon_i, \eta_i=1$ for all $i$ we get
\begin{align*}
\|x+1_{\varepsilon A}\|={\frac{n}{2}}\left(\frac{3n}{2}+1\right)
\end{align*}
and
\begin{align*}
\|x+1_{\eta B}\|=2n.
\end{align*}
Thus $(e_n)$ does not satisfies left Property $(A)$.

\end{ex}

We now give characterization of $w$-partially greedy bases in terms
of quasi-greedy property and $w$-left Property $(A)$.
\begin{Prop}\label{prop:partially implies qlpa}
Let $(e_n)$ be a basis of a Banach space $X$. If $(e_n)$ is $C_p$-$w$-partially greedy,
then $(e_n)$ is $(C_p+1)$-quasi-greedy and has
$C_p$-$w$-left Property $(A)$.
\end{Prop}

\begin{proof}
Let $(e_n)$ be a $w$-partially greedy basis with constant $C_p$.
Thus
\begin{align*}
\|x-G_m(x)\|\leq C_p \inf\{ \| x-P_A (x) \|:A\in\N^{<\infty},
~w(A)\leq w(\Lambda_m(x)),~A<\al_m(x)\}
\end{align*}
for all $x\in X$ and $m\in \N$.

If we take $A=\emptyset$, then we get $\|x-G_m(x)\|\leq C_p \|x\|$
and hence $\|G_m(x)\|\leq (C_p+1)\|x\|$ for all $x\in X$, $m\in \N$.

Let $x\in X$, $t\geq \sup_j |e_j^*(x)|$, $A,B\in \N^{<\infty}$ with
$A<B$, $w(A)\leq w(B)$, $supp(x)\cap (A\cup B)=\emptyset$ and
$\varepsilon,\eta\in \mathcal{A}$. If we take $y=x+t1_{\varepsilon
A}+(t+\delta)1_{\eta B}$ for any $\delta>0$, then
$\|x+t1_{\varepsilon A}\|=\|y-G_{|B|}(y)\|\leq
C_p\|y-t1_{\varepsilon A}\|=\|x+(t+\delta)1_{\eta B}\|$. This gives
that $(e_n)$ satisfies $C_p$-$w$-left Property $(A)$.
\end{proof}

From Remark~\ref{Aimpliesconservative} it is clear that  $w$-left
Property $(A)$ implies $w$-conservative property. Thus from
Theorem~\ref{pg} it follows that if a basis $(e_n)$ is quasi-greedy
and has $w$-left Property $(A)$ then $(e_n)$ is $w$-partially greedy. We
now prove this result using the arguments similar to
\cite{BDKOW} to get better estimates in terms of the constant.

\begin{Prop}\label{proertyAalternative}
A basis $(e_n)$ of a Banach space $X$ satisfies $w$-left left
Property $(A)$ with constant $C_{la}$ if and only if
\begin{align}\label{wlpa2}
\|x\|\leq C_{la} \|x-P_A(x)+1_{\eta B}\|
\end{align}
where $x\in X$, $\sup_j |e_j^*(x)|\leq 1$, $A,B\in \N^{<\infty}$,
$A<B$, $w(A)\leq w(B)$, $supp(x)\cap  B=\emptyset$ and $\eta\in \mathcal{A}$.
\end{Prop}
\begin{proof}
First observe that \cite[Lemma 2.4]{BDKOW} implies for the proof of
the result it is sufficient to consider only finitely supported
$x\in X$.

Let $(e_n)$ has $w$-left Property $(A)$ with constant $C_{la}$, $x\in
X$, $\sup_j|e_j^*(x)|\leq 1$, $A,B\in \N^{<\infty}$, $A<B$, $w(A)\leq
w(B)$, $supp(x)\cap  B=\emptyset$ and $\varepsilon, \eta\in \mathcal{A}$. From
the $w$-left Property $(A)$ of the basis we can write
 \begin{align*}
 \|P_{A^c}(x)+1_{\varepsilon A}\|&\leq C_{la}\|P_{A^c}(x)+1_{\eta B}\|\\
 &=C_{la}\|x-P_{A}(x)+1_{\eta B}\|.
 \end{align*}
Since $x$ belongs to the convex hull of $\{P_{A^c}(x)+1_{\varepsilon A}:\varepsilon \in \mathcal{A}\}$, so we get
 \begin{align*}
 \|x\|\leq C_{la}\|x-P_{A}(x)+1_{\eta B}\|.
 \end{align*}
 Conversely, for any  $x\in X$ with $\sup_j |e_j^*(x)|\leq 1$, $A,B\in \N^{<\infty}$, $A<B$, $w(A)\leq w(B)$, $supp(x)\cap  (A\cup B)=\emptyset$ and $\varepsilon,\eta\in \mathcal{A}$,
 consider $y=x+1_{\varepsilon A}$. Now from $(\ref{wlpa2})$ we get
  \begin{align*}
\|y\|= \|x+1_{\varepsilon A}\|&\leq C_{la}\|y-P_{A}(y)+1_{\eta B}\|\\
 &=C_{la}\|x+1_{\eta B}\|.
 \end{align*}
 \end{proof}

Similarly we can prove the following results for a basis with
$w$-right Property $(A)$.
\begin{Prop}
A basis $(e_n)$ of a Banach space $X$ satisfies $C_{ra}$-$w$-right Property $(A)$
if and only if
\begin{align}
\|x\|\leq C_{ra} \|x-P_A(x)+1_{\eta B}\|
\end{align}
where $x\in X$, $\sup_j |e_j^*(x)|\leq 1$, $A,B\in \N^{<\infty}$,
$B<A$, $w(A)\leq w(B)$, $supp(x)\cap B=\emptyset$ and $\eta\in \mathcal{A}$.
\end{Prop}

Before proving the next result we first recall the following result
from \cite{BDKOW}.

\begin{lem}\label{lem:truncation}
Let $(e_n)$ be a $C_q$-quasi-greedy basis of a Banach space $X$,
$x\in X$ and $\lambda>0$. If $T_\lambda$ is an operator defined on
$X$ by
\begin{align*}
T_\lambda(x)=\sum_{i \in A} \lambda sgn(e_i^*(x))e_i+\sum_{i \in
A^c}e_i^*(x)e_i
\end{align*}
where $A=\{i: \lambda<|e_i^*(x)|\}$, then $\|T_\lambda(x)\|\leq
(C_q+1)\|x\|$.
\end{lem}

\begin{thm}
Let $(e_n)$ be a basis of a Banach space $X$. \bla
\item If $(e_n)$ is $C_p$-$w$-partially greedy, then $(e_n)$ is $(C_p+1)$-quasi-greedy and has $C_p$-$w$-left Property $(A)$.
\item  If $(e_n)$ has $C_{la}$-$w$-left Property $(A)$ and is $C_q$-quasi-greedy then $(e_n)$ is $(C_q+1)C_{la}$-$w$-partially greedy.
\el
\end{thm}

\begin{proof}
$(a)$ follows from  Proposition~\ref{prop:partially implies qlpa}.

Let $x\in X$ and $\varepsilon>0$. Now, choose $A\in \N^{< \infty}$
with $A<\Lambda_m(x),~w(A)\leq w(\Lambda_m(x))$ such that
$\|x-P_A(x)\|<
{\stackrel{\sim}\sigma}_{w(\Lambda_m(x))}^L(x)+\varepsilon$. Fix
$t=|e_{\rho(m)}^*(x)|$ and $\eta \in \mathcal{A}$ with $\eta_j=
sgn(e_j^*(x))$ for all $j$, then from
Proposition~\ref{proertyAalternative} and Lemma~\ref{lem:truncation}
it follows that
\begin{align*}
\|x-G_m(x)\|&\leq C_{la}\|x-P_{\Lambda_m(x)}-P_A(x)+t1_{\eta \Lambda_m(x)}\|\\
             &=C_{la}\|P_{(A\cup \Lambda_m(x))^c}(x-P_A(x))+t1_{\eta \Lambda_m(x)} \|\\
             &=C_{la}\|T_t(I-P_A)(x)\|\\
             &\leq (C_q+1)C_{la}\|x-P_A(x)\|\\
             &< (C_q+1)C_{la}({\stackrel{\sim}\sigma}_{w(\Lambda_m(x))}^L(x)+\varepsilon).
\end{align*}
From this we can conclude that $\|x-G_m(x)\|\leq (C_q+1)C_{la}
{\stackrel{\sim}\sigma}_{w(\Lambda_m(x))}^L(x)$ and thus the basis
is $(C_q+1)C_{la}$-$w$-partially greedy.
\end{proof}

In a similar fashion we can prove the following characterization of  $w$-reverse
partially greedy basis.

\begin{thm}
Let $(e_n)$ be a basis of a Banach space $X$. \bla
\item If $(e_n)$ is $C_{rp}$-$w$-reverse partially greedy, then $(e_n)$ is $(C_{rp}+1)$-quasi-greedy and has $C_{rp}$-$w$-left Property $(A)$.
\item If $(e_n)$ has $C_{ra}$-$w$-right Property $(A)$ and is $C_q$-quasi-greedy then $(e_n)$ is $(C_q+1)C_{ra}$-$w$-reverse partially greedy.
\el
\end{thm}

We now prove that if a basis $(e_n)$ of a Banach space $X$ satisfies
both left Property $(A)$ and right Property $(A)$ then the basis
satisfies Property $(A)$. If $X$ is an infinite-dimensional Banach
space then for given $A,B\in \N^m$ we can find $C \in \N^m$ with
$A<C$ and $B<C$. Now by the left Property $(A)$ and right Property $(A)$
of the basis it can be easily proved that the basis satisfies
Property $(A)$. We now give another proof of this fact which has the
advantage of  working  for a \textit{finite} basis as well.

\begin{lem}\label{lem:property A}
Let $(e_n)$ be a basis of Banach space $X$. If $(e_n)$ satisfies
both left Property $(A)$ and right Property $(A)$ then the basis
satisfies Property $(A)$.
\end{lem}
\begin{proof}
Let  $C$ be a constant such that for any $x\in X$, $t\geq \sup_n
|e_n^*(x)|,~\varepsilon,\eta\in \mathcal{A}$,  $A,B\in
\N^{<\infty}$, $supp(x)\cap (A\cup B)=\emptyset$ and $|A|\leq |B|$
we have
\begin{align*}
\|x+t1_{\varepsilon A}\|\leq C\|x+t1_{\eta B}\|
\end{align*}
where either $A<B$ or $B>A$.

Observe that to prove Property $(A)$ it is
sufficient to consider any two disjoint sets  $A,B$ of same cardinality.
Choose any $x\in X$, $t\geq \sup_n |e_n^*(x)|,~\varepsilon,\eta\in \mathcal{A}$  and any two disjoint sets  $A,B\in \N^m$ with
$supp(x)\cap (A\cup B)=\emptyset$.  Let $A=\{a_1<a_2<\ldots <a_m\}$
and $B=\{b_1<b_2<\ldots <b_m\}$. If $A <  B$ or $B< A$, then
$\|x+t1_{\varepsilon A}\|\leq C\|x+t1_{\eta B}\|$ and $\|x+t1_{\eta
B}\|\leq C\|x+t1_{\varepsilon A}\|$. So, for the rest of the proof
we assume that this is not the case.

Without loss of generality we can assume that $a_1<b_m$. If we
compare $a_2$ with $b_{m-1}$ then there can be two possibilities:
$a_2<b_{m-1}$ or $a_2>b_{m-1}$. If $a_2>b_{m-1}$ then we will stop
the process; otherwise we will continue in the same manner. By the
assumptions on the sets $A,B$ we can find the first $j$, $1\leq j <
m$, such that $a_{j+1}>b_{m-j}$.

Thus we can write

$A=A_1\cup A_2$ and $B=B_1\cup B_2$ where $|A_i|=|B_i|$, $A_1<B_1$,
$A_2>B_2$,  $A_1=\{a_1<\ldots<a_j\}<B_1=\{b_{m-j+1}<\ldots<b_m\}$
and $A_2=\{a_{j+1}<\ldots<a_m\}>B_2=\{b_{1}<\ldots<b_{m-j}\}$

Let $x_1=\sum_{i>min B_1} e_i^*(x) e_i$, $x_2=x-x_1$,
$y_1=\sum_{i<max A_1} e_i^*(x) e_i$ and $y_2=x-y_1$. Then

\begin{align*}
\|x_i+t1_{\varepsilon A_i}\|\leq C\|x_i+t1_{\eta B_i}\|,~~  i=1,2.
\end{align*}

and
\begin{align*}
\|y_i+t1_{\eta B_i}\| \leq C\|y_i+1_{\varepsilon A_i}\| ,~~  i=1,2.
\end{align*}

Now we can write
\begin{align*}
\|x+t1_{\varepsilon A}\|&\leq \|x_1+t1_{\varepsilon A_1}\|+
\|x_2+t1_{\varepsilon A_2}\|\\
&\leq C(\|x_1+t1_{\eta B_1}\|+\|x_2+t1_{\eta B_2}\|)\\
&\leq C (2K_b+1) \|x+t1_{\eta B}\|
\end{align*}
and
\begin{align*}
\|x+t1_{\eta B}\|&\leq \|y_1+t1_{\eta B_1}\| + \|y_2+t1_{\eta
B_2}\|\\
&\leq C(\|y_1+t1_{\varepsilon A_1}\|+\|y_2+t1_{\varepsilon
A_2}\|)\\
&\leq C (2K_b+1) \|x+t1_{\varepsilon A}\|
\end{align*}
where $K_b$ is the basis constant for $(e_n)$.
\end{proof}

Consider any weight sequence $w=(w_n)$ bounded away from 0 and
$\infty$. For such a weight sequence we now prove that any basis
satisfying both $w$-left Property $(A)$ and $w$-right Property $(A)$ satisfies
$w$-Property $(A)$. If $X$ is an infinite-dimensional Banach space
then this result can be proved by a simple argument. But we will present a proof
which works for both finite and infinite-dimensional Banach spaces.

To prove this result we first extend the idea used in
\cite[Proposition 4.9]{DKTW} to prove that in this case any basis
$(e_n)$ has $w$-left Property $(A)$ ($w$-right Property $(A)$) if and
only if $(e_n)$ has left Property $(A)$ (right Property $(A)$).

\begin{Prop}\label{prop:left property}
Let $w$ be a weight sequence with $0<\inf w_n\leq \sup w_n <\infty$.
Then $(e_n)$ has $w$-left Property $(A)$ if and only if $(e_n)$ has left
Property $(A)$.
\end{Prop}

\begin{proof}
For simplicity we assume that $0<\al=\inf w_n\leq \sup w_n =1$.

Let $(e_n)$ satisfies $w$-left Property $(A)$ with constant $C_{la}$.
Let $x\in X$, $t\geq \sup_n |e_n^*(x)|,~\varepsilon,\eta\in \mathcal{A}$, $A,B\in \N^{<\infty}$ with $A<B$, $|A|=|B|$ and $supp(x)\cap
(A\cup B)=\emptyset$.

If $w(A)\leq w(B)$ then from the $w$-left Property $(A)$  of
the basis we get
\begin{align*}
\|x+t1_{\varepsilon A}\|\leq C_{la} \|x+t1_{\eta B}\|.
\end{align*}

Now consider the case when $w(B)\leq w(A)$. If in this case $w(A)
\leq {\frac{2}{\al}}$ then $|A|,~|B| \leq {\frac{2}{\al^2}}$ and
thus
\begin{align*}
\|x+t1_{\varepsilon A}\|&\leq \|x+t1_{\eta B}\|+\|t1_{\varepsilon A}\|+\|t1_{\eta B}\|\\
&\leq  \|x+t1_{\eta B}\|+{\frac{4}{\al^2}}t\\
&=  \|x+t1_{\eta B}\|+{\frac{4}{\al^2}}\|t1_{\eta b}\|\\
&\leq  \|x+t1_{\eta B}\|+{\frac{8K_b}{\al^2}}\|x+t1_{\eta B}\|
\end{align*}
where $K_b$ is the basis constant of $(e_n)$ and $b\in B$.

If $w(A)>{\frac{2}{\al}}$ then $2<\al w(A) \leq \al |A|=\al|B|\leq
w(B)$. Thus we can partition $A$ into $N$ sets $A_1,\ldots,A_N$
satisfying $w(A_i)\leq w(B)\leq w(A_i)+1$, hence $w(A_i)\geq
{\frac{w(B)}{2}}$ and
\begin{align*}
N\leq {\frac{w(A)}{\frac{w(B)}{2}} }\leq {\frac{2}{\al}}.
\end{align*}

Now by combining the facts that $A_i<B$, $w(A_i)\leq w(B)$ with the
$w$-left Property $(A)$ of the basis we get
\begin{align*}
\|x+t1_{\varepsilon A}\|&\leq \sum_1^N\|{\frac{x}{N}}+t1_{\varepsilon A_i}\|\\
          &\leq C_{la}N\|{\frac{x}{N}}+t1_{\eta B}\|\\
          &\leq C_{la}(N\|x+t1_{\eta B}\|+(N-1)\|x\|)\\
          &\leq C_{la}(N\|x+t1_{\eta B}\|+C_{la}(N-1)\|x+t1_{\eta B}\|)\\
          &\leq {\frac {4C_{la}^2}{\al}}\|x+t1_{\eta B}\|
\end{align*}

and this proves that the basis satisfies left Property $(A)$.

Let $(e_n)$ satisfies left Property $(A)$ with constant $C$. Choose $x\in X$,
$t\geq \sup_n |e_n^*(x)|,~\varepsilon,\eta\in \mathcal{A}$,  $A,~B\in
\N^{<\infty}$ with $A<B$, $w(A)\leq w(B)$ and $supp(x)\cap (A\cup
B)=\emptyset$. Then $|A|\leq {\frac{1}{\al}}|B|$ and we can
partition $A$ into $N$ sets $A_1,\ldots,A_N$  with $|A_i|\leq |B|$,
$A_i<B$ and $N\leq 1+ {\frac{1}{\al}}$.

Now from the left Property $(A)$ of the basis we get
\begin{align*}
\|x+t1_{\varepsilon A}\|&\leq \sum_1^N\|{\frac{x}{N}}+t1_{\varepsilon A_i}\|\\
          &\leq CN\|{\frac{x}{N}}+t1_{\eta B}\|\\
          &\leq C(N\|x+t1_{\eta B}\|+(N-1)\|x\|)\\
          &\leq C(N\|x+t1_{\eta B}\|+C(N-1)\|x+t1_{\eta B}\|)\\
          &\leq 2C^2({1+{\frac{1}{\al}}})\|x+t1_{\eta B}\|.
\end{align*}
Thus $(e_n)$ satisfies $w$-left Property $(A)$.

\end{proof}

\begin{Prop}\label{prop:right property}
Let $w$ be a weight sequence with $0<\inf w_n\leq \sup w_n <\infty$.
Then $(e_n)$ satisfies $w$-right Property $(A)$ if and only if $(e_n)$
satisfies right Property $(A)$.
\end{Prop}

\begin{thm}
Let $w$ be a weight sequence with $0<\inf w_n\leq \sup w_n <\infty$.
If a basis $(e_n)$ satisfies both $w$-left Property $(A)$ and $w$-right
Property $(A)$ then the basis satisfies $w$-Property $(A)$.
\end{thm}

\begin{proof}
Let $0< \al=\inf w_n \leq \sup w_n=1$ and $(e_n)$ satisfies both
$w$-left Property $(A)$, $w$-right Property $(A)$. Then form
Propositions~\ref{prop:left property},~\ref{prop:right property} and
Lemma~\ref{lem:property A},  it follows that the basis satisfies
Property $(A)$. Let $C$ be the constant for Property $(A)$.

Let $x\in X$, $t\geq \sup_j |e_j^*(x)|$,  $A,B\in \N^{<\infty}$ with
$A\cap B=\emptyset$, $w(A)\leq w(B)$, $supp(x)\cap (A\cup
B)=\emptyset$ and $\varepsilon,\eta\in \mathcal{A}$. $w(A)\leq w(B)$ gives $|A|\leq
{\frac{1}{\al}}|B|$ and thus we we can partition $A$ into $N$ sets
$A_1,\ldots,A_N$  with $|A_i|\leq |B|$, $N\leq 1+ {\frac{1}{\al}}$.
Using the arguments similar to Proposition~\ref{prop:left property}
we get $\|x+1_{\varepsilon A}\|\leq
2C^2({1+{\frac{1}{\al}}})\|x+t1_{\eta B}\|$ and this proves that the
basis satisfies $w$-Property $(A)$.
\end{proof}

Next, we will prove the similar result for a weight sequence $(w_n)$ with
$\sum w_n=\infty$ and $\sup w_n<\infty$. We will consider some other conditions
on the weight sequences in the next section for any basis satisfying weaker conditions than $w$-left Property $(A)$ and $w$-right
Property $(A)$.

\begin{thm}\label{infinitesum}
Let $w$ be a weight sequence with $\sum w_n=\infty$ and $\sup
w_n<\infty$. If a basis $(e_n)$ satisfies both $w$-left Property $(A)$
and $w$-right Property $(A)$ then $(e_n)$ satisfies $w$-Property $(A)$.

\end{thm}
\begin{proof}
Observe that from \cite[Lemma 2.4]{BDKOW} it follows that for the proof of this result it
 is sufficient to consider $x\in X$ with $supp(x)<\infty$.
Let $x\in X$, $t\geq \sup_n |e_n^*(x)|,~\varepsilon,\eta,\gamma \in \mathcal{A}$, $A,B\in \N^{<\infty}$ with $w(A)\leq w(B)$, $A\cap B=\emptyset$ and
$supp(x)\cap (A\cup B)=\emptyset$. Let $C$ be the constant for
$w$-left Property $(A)$ and $w$-right Property $(A)$.

If $w(B)\leq \limsup_{n\rightarrow \infty} w_n$ then from
Proposition~\ref{Prop:c0 basis} it follows that $\|1_{\varepsilon
A}\|, \|1_{\eta B}\|\leq 4C$. Thus we get
\begin{align*}
\|x+t1_{\varepsilon A}\|&\leq \|x+t1_{\eta B}\|+\|t1_{\varepsilon A}\|+\|t1_{\eta B}\|\\
&\leq  \|x+t1_{\eta B}\|+8Ct\\
&=  \|x+t1_{\eta B}\|+8C\|t1_{\eta b}\|\\
&\leq  \|x+t1_{\eta B}\|+16CK_b\|x+t1_{\eta B}\|
\end{align*}
where $K_b$ is the basis constant of $(e_n)$ and $b\in B$.

If $w(B)>\limsup_{n\rightarrow \infty} w_n$ then $\sum w_n=\infty$
implies that  we can choose $E\in \N^{<\infty},~n\in\N$ such that $F=\{n\cup
E\}>(A\cup B)$, $w(E)\leq w(B)\leq w(F)$ and $supp(x)\cap F=\emptyset$. Observe that
$w(A)\leq w(F)$ thus from the $w$-left Property $(A)$ and $w$-right
Property $(A)$ of the basis we can write
\begin{align*}
\|x+t1_{\varepsilon A}\|&\leq C \|x+t1_{\gamma F}\|\\
                         &\leq C\{ \|x+t1_{\gamma E}\|+t\|e_n\|\}\\
                         &\leq C\{ C\|x+t1_{\eta B}\|+2K_b\|x+t1_{\eta B}\|\}\\
                         &=(C^2+2CK_b)\|x+t1_{\eta B}\|.
\end{align*}
\end{proof}

\section{weight-conservative and weight-reverse conservative bases}

First, we will prove the following improvement of \cite[Proposition
3.10]{BDKOW}. The proof of \cite[Proposition 3.10]{BDKOW} itself
suggests that we do not need the full power of $w$-superdemocracy of
the basis (see \cite [Definition 3.6]{BDKOW}). The canonical  basis
of the space $c_0$, the space of all real sequences converging to
$0$, is $(e_n)$ where $e_n$ has only one non zero entry at the
$n$-th place and this entry has value equal to $1$.

\begin{Prop}\label{Prop:c0 basis}
Let $(e_n)$ be a basis of a Banach space X. \bla
\item  If $(e_n)$ is $w$-conservative basis with constant $C$ and $A\in \N^{<\infty}$ with
$w(A)\leq \limsup_{n\rightarrow \infty} w_n$ then $max_{\varepsilon \in \mathcal{A}}\|1_{\varepsilon A}\| \leq 4 C$.

\item If $\sup w_n=\infty$ and $(e_n)$ is $w$-conservative basis then $(e_n)$ is equivalent to the canonical  basis of $c_0$.

\item If $\sum w_n<\infty$ and $(e_n)$ is $w$-reverse conservative basis then $(e_n)$ is equivalent to the canonical  basis of $c_0$.

\item If $\inf w_n=0$ and $(e_n)$  $w$-reverse conservative basis then $(e_n)$  contains a subsequence which is equivalent to the canonical  basis of $c_0$.

\el
\end{Prop}

\begin{proof}
\bla
\item For any $A\in \N^{<\infty}$ with $w(A)\leq \limsup_{n\rightarrow \infty} w_n$, we can choose $n_1>n_0>A$ such that
$w(A)<w_{n_0}+w_{n_1}$. For any $\varepsilon \in \mathcal{A}$ we can
break $A$ into two parts $A^+$ and $A^-$ where $\varepsilon_i=1$ for
all $i\in A^+$ and $\varepsilon_i=-1$ for all $i\in A^-$. Since
$A^+,~A^-<\{n_0,n_1\}$ and $w(A^+),w(A^-)\leq w_{n_0}+w_{n_1}$ we
get $\|1_{\varepsilon A}\|\leq \|1_{A^+}\|+\|1_{A^-}\| \leq
2C\|e_{n_0}+e_{n_1}\|\leq 4C.$

\item For any given $A\in \N^{<\infty}$ we can find $n>A$ with $w(n)>w(A)$. Thus for any $\varepsilon \in \mathcal{A}$  we get $\|1_{\varepsilon A}\|\leq 2C\|e_n\|\leq 2C.$

\item Choose $n\in \N$ such that $\sum_{n+1}^{\infty}w_i<w_1$. If $A\in \N^{<\infty}$ with $A>n$ then $A>1$ and $w(A)<w_1$. Now from the $w$-reverse conservative property
we can write $\|1_{\varepsilon A}\|\leq 2C\|e_1\|\leq 2C$ for all
$\varepsilon \in \mathcal{A}$.

\item Choose $(n_i)_i$ such that $\sum_iw_{n_i}<\infty$. From $(c)$ it follows that $(e_{n_i})$ is equivalent to the canonical unit vector basis of $c_0$.
\el
\end{proof}

In \cite{DKTW} the authors proved that if $w\in c_0$ then any $w$-almost
greedy basis is weakly null. We now prove the similar result for
$w$-reverse conservative basis.

\begin{cor}
If $w\in c_0$ then any $w$-reverse conservative basis is weakly
null.
\end{cor}

\begin{proof}
First observe that if $w\in c_0$ then any subsequence $(w_{n_k})$
contains a further sequence $(w_{n_{k_i}})$ such that
$\sum_iw_{n_{k_i}}<\infty$.  Also if $(e_n)$ is $w$-reverse
conservative basis then $(e_{n_k})$ is $w'$-reverse conservative where $w'=(w_{n_k})$.

Let $w\in c_0$ and $(e_n)$ be a $w$-reverse conservative basis of
$X$ which is not weakly null. Then there exists a $f\in X^*$ such
that $f(x_n)\not\longrightarrow 0$. We can find a subsequence
$(e_{n_k})$ such that $|f(e_{n_k})|\longrightarrow \de >0$. This
implies that $(e_{n_k})$ doesn't contains any weakly null sequence
and thus there is no further subsequence of $(e_{n_k})$ which is
equivalent to the canonical unit vector basis of $c_0$.
Proposition~\ref{Prop:c0 basis}(c) implies that this contradicts
the $w$-reverse conservative property of $(e_n)$.
\end{proof}

\begin{Prop}
Let $w\in c_0$ and $(e_n)$ is both $w$-reverse conservative and
conservative. Then $(e_n)$ is equivalent to the canonical basis of $c_0$.
\end{Prop}
\begin{proof}
Let $A\in \N^{<\infty}$ and $C_1$, $C_2$ be the constant for
$w$-reverse conservative and conservative property respectively. Since $w\in
c_0$, we can find a subsequence $(e_{n_i})$ such that $\sum_i
e_{n_i}<\infty$. Choose $N>A$ such that
$\sum_{N+1}^{\infty}w_{n_i}<w_1$. Now we can find a finite set $D\subset \{n_i:i\geq N+1\}$
with $|D|\geq |A|$ and $w(D)<w_1$.

From the conservative and $w$-reverse conservative property we have
$\|1_A\|\leq C_2\|1_D\|\leq C_1C_2\|e_1\|=C_1C_2$.

Thus for any $A\in \N^{<\infty}$ and $\varepsilon \in \mathcal{A}$ we
have $\|1_{\varepsilon A}\|\leq 2C_1C_2$ and this concludes the
proof.
\end{proof}

Similar arguments as in section $3$ yields the following results.
\begin{thm}
Let $w$ be a weight sequence with $\sum w_n=\infty$ and $\sup
w_n<\infty$. If a basis $(e_n)$ is both $w$-conservative and
$w$-reverse conservative then  $(e_n)$ is $w$-democratic.
\end{thm}

\begin{Prop}\label{prop:equivalence cr}
Let $w$ be a weight sequence with $0<\inf w_n\leq \sup w_n <\infty$.
Then basis $(e_n)$ of a Banach space X is $w$-conservative
($w$-reverse conservative) if and only if $(e_n)$ is conservative
(reverse conservative).
\end{Prop}

\begin{thm}
Let $w$ be a weight sequence with $0<\inf w_n\leq \sup w_n <\infty$.
If a basis $(e_n)$ is both $w$-conservative and $w$-reverse
conservative then  $(e_n)$ is $w$-democratic.
\end{thm}

\begin{proof}
Let $(e_n)$ is both $w$-conservative, $w$-reverse conservative and
$0<\al=\inf w_n \leq\sup w_n=1$. From
Proposition~\ref{prop:equivalence cr} it follows that the basis is
both conservative and reverse conservative. Now from \cite[Lemma
2.8]{DK} it follows that the basis is democratic. Let $C$ be the
constant of democracy.

Let $A,~B\in \N^{<\infty}$ with $w(A)\leq w(B)$. Thus $|A|\leq
{\frac{1}{\al}}|B|$ and we can write $A=\cup_{1}^N A_i$ where
$|A_i|\leq |B|$ and $N\leq 1+ {\frac{1}{\al}}$. Now from the
democracy property of the basis we get
\begin{align*}
\|1_A\|\leq \sum_1^N\|1_{A_i}\|\leq CN\|1_B\|\leq
C(1+{\frac{1}{\al}})\|1_B\|
\end{align*}
and this proves that the basis is $w$-democratic.
\end{proof}

\section*{Acknowledgements}
The author is thankful to Gideon Schechtman for numerous useful discussions.


\begin{thebibliography}{99}

\bibitem{AW}F. Albiac and P.Wojtaszczyk, {\em Characterization of
1-greedy basis}, J. Approx. Theory {\bf 138} (2006), no. 1, 65--86.

\bibitem{BDKOW} P.M. Berna, S.J. Dilworth, D. Kutzarova, T. Oikhberg and B. Wallis, {\em The weighted Property $(A)$ and greedy algorithms}, arXiv:1803.05052.

\bibitem{S}S.J. Dilworth, N.J. Kalton, D. Kutzarova and V.N. Temlyakov, {\em The thresholding greedy algorithm, greedy basis, and duality}, Constr. Approx.
{\bf 19} (2003), no. 4, 575--597.

\bibitem{DK}S.J. Dilworth and Divya Khurana, {\em Characterizations of almost greedy and partially greedy bases}, arXiv:1805.06778.

\bibitem{DKTW} S.J. Dilworth, D. Kutzarova, V.N. Temlyakov and B. Wallis, {\em Weighted almost greedy bases}, arXiv:1803.02932.


\bibitem{KPT} G. Kerkyacharian, D. Picard and V.N.Temlyakov, {\em Some inequlities for the tensor product greedy bases and weight-greedy bases}, East. J. Approx. {\bf 12} (2006), 103--118.

\bibitem{KT}S.V. Konyagin and V.N.Temlyakov, {\em A remark on greedy
approximation in Banach spaces}, East. J. Approx. {\bf 5} (1999),
365--379.

\bibitem{W}P. Wojtaszczyk, {\em Greedy algorithm for general
biorthogonal systems}, J. Approx. Theory {\bf 107} (2000), 293--314.



\end{thebibliography}
\end{document}